\documentclass[12pt]{amsart}
\usepackage{amsthm,amsmath,amsfonts,amssymb}
\usepackage[a4paper]{geometry}
\usepackage{latexsym}
\usepackage{enumerate}
\usepackage{url}
\usepackage{graphicx,graphics}
\theoremstyle{plain}
\newtheorem{thm}{Theorem}[section]

\newtheorem{lem}[thm]{Lemma}

\theoremstyle{definition}
\newtheorem{defn}[thm]{Definition}

\theoremstyle{remark}

\begin{document}

\title[Delicacy of the RH]{Delicacy of the Riemann hypothesis and certain subsequences of superabundant numbers}
\author{Sadegh Nazardonyavi, Semyon Yakubovich}
\email{sdnazdi@yahoo.com\vspace*{10pt}}

\email{syakubov@fc.up.pt}
\maketitle
\begin{abstract}
Robin's theorem is one of the ingenious reformulation of the Riemann hypothesis (RH). It states that the RH is true if and only if $\sigma(n)<e^\gamma n\log\log n$ for all $n>5040$ where $\sigma(n)$ is the sum of divisors of $n$ and $\gamma$ is Euler's constant. In this paper we show that how the RH is delicate in terms of certain subsets of superabundant numbers, namely extremely abundant numbers and some of its specific supersets.
\end{abstract}
\section{Introduction}
Let $\sigma(n)$ be the sum of divisors of a positive integer $n$. Gronwall \cite{GR} showed that the order of $\sigma(n)$ is very nearly $n$. More precisely

\begin{equation}\label{GR}
    \limsup_{n\rightarrow\infty}\frac{\sigma(n)}{n\log\log n}=e^\gamma.
\end{equation}

In 1984, Robin \cite{Robin} established an elegant problem equivalent to the RH which is stated in the following theorem.
\begin{thm}[\cite{Robin}]
The Riemann hypothesis is equivalent to
\begin{equation}\label{Rob}
\sigma(n)<e^\gamma n\log\log n,\qquad\text{for all }\ n>5040.
\end{equation}
where $\gamma\approx0.57721566$ is Euler's constant.
\end{thm}
Inequality (\ref{Rob}) is called Robin's inequality. He also showed that for all $n\geq3$

\begin{equation}\label{Rob 0.6482}
\frac{\sigma(n)}{n}\leq e^\gamma\log\log n+\frac{0.648214}{\log\log n},
\end{equation}
where $0.648214\approx(\frac73-e^\gamma\log\log12)\log\log12$ with equality for $n=12$.

A positive integer $n$ is called superabundant number ((\cite{Al}, see also \cite{Ram.N})) if
$$
\frac{\sigma(n)}{n}>\frac{\sigma(m)}{m},\qquad\text{for all }\ m<n.
$$
and it is called colossally abundant, if for some $\varepsilon>0$,
\begin{equation*}\label{CA defn}
\frac{\sigma(n)}{n^{1+\varepsilon}}\geq\frac{\sigma(m)}{m^{1+\varepsilon}},\qquad\text{for all }m>1.
\end{equation*}
Robin also proved that if the RH is not true, then for colossally abundant numbers we have
\begin{equation}\label{Robin oscillation}
\frac{\sigma(n)}{n\log\log n}=e^\gamma(1+\Omega_\pm(\frac{1}{\log^b n}))
\end{equation}
where $b$ is any number of the interval $(1-\theta, 1/2)$, and $\theta$ being the upper bound of
real parts of the zeros of the Riemann zeta function.

In 2009, Akbary and Friggstad \cite{Ak} proved that to find the first probable counterexample to Robin's inequality it is enough to look for a special subsequence of positive integers. Indeed, If there is any counterexample to (\ref{Rob}), then the least such counterexample is a superabundant number (cf.\cite{Ak}).\\

During the study of Robin's theorem and looking for the first probable integer which violates (\ref{Rob}) and belongs to a proper subset of superabundant numbers, authors in \cite{Sadegh} constructed a new subsequence of positive integers (called its elements extremely abundant numbers) via the following definition:
\begin{defn}\label{extn}
A positive integer $n$ is an \emph{extremely abundant} number, if either $n=10080$ or
\begin{equation}\label{extn1}
\forall m\quad\text{\upshape{s.t.}} \quad10080\leq m<n,\qquad\frac{\sigma(m)}{m\log\log m}<\frac{\sigma(n)}{n\log\log n}.
\end{equation}
\end{defn}
We denote the following sets of integers by
\begin{align*}
  S =&\ \{n:\ \ n \mbox{\ is superabundant}\}, \\
  X =&\ \{n:\ \ n \mbox{\ is extremely abundant}\}.
\end{align*}
We also use SA and XA as abbreviations of superabundant and extremely abundant, respectively. It can be shown  that $X\subset S$ (see \cite{Sadegh}). Combining Gronwall's theorem with Robin's theorem,  authors \cite{Sadegh} established the following interesting results:\\[-8pt]
\begin{enumerate}[(i)]
  \item If there is any counterexample to Robin's inequality, then the least one is an XA number.
  \item The RH is true if and only if $X$ is an infinite set.
\end{enumerate}
The statement (ii) is the first step for showing the delicacy of RH.\\

\begin{defn}
Let $n_1=10080$. We find $n_2$ such that
\begin{equation*}\label{X' defn}
\frac{\sigma(n_2)/n_2}{\sigma(n_1)/n_1}>1+\frac{\log n_2/n_1}{\log n_2\log\log n_1}.
\end{equation*}
Now we find $n_3$ such that
\begin{equation*}\label{X' defn}
\frac{\sigma(n_3)/n_3}{\sigma(n_2)/n_2}>1+\frac{\log n_3/n_2}{\log n_3\log\log n_2},
\end{equation*}
and so on. We define $X'$ to be the set of all $n_1,\ n_2,\ n_3,\ldots$.
\end{defn}

\begin{equation}\label{X<X'<S}
X\subset X'\subset S.
\end{equation}
\begin{lem}
If $m\in X'$, then there exists $n>m$ such that

\begin{equation}\label{s(n)/n>s(m)/m(1+log(n/m)/(log n loglog m))}
\frac{\sigma(n)/n}{\sigma(m)/m}>1+\frac{\log n/m}{\log n\log\log m}.
\end{equation}
\end{lem}
\begin{proof}
Given $m\in X'$. Then by (\ref{Rob 0.6482})

\begin{equation}\label{Rob 0.6482 m}
\frac{\sigma(m)}{m}\leq \left(e^\gamma+\frac{0.648214}{(\log\log m)^2}\right)\log\log m,
\end{equation}
Since
$$
\frac{\log\log m}{\log\log m'}\left(1+\frac{\log m'/m}{\log m'\log\log m}\right)<1
$$
and decreasing for $m'>m$ and tends to 0 as $m'$ goes to infinity, then for some $m'>m$ we have
\begin{equation}\label{loglog m/loglog m'...=ey-eps}
\frac{\log\log m}{\log\log m'}\left(1+\frac{\log m'/m}{\log m'\log\log m}\right)\left(e^\gamma+\frac{0.648214}{(\log\log m)^2}\right)=e^\gamma-\varepsilon,
\end{equation}
where $\varepsilon>0$. Hence by Gronwall's theorem there is $n\geq m'$ such that
\begin{align*}
\frac{\sigma(n)}{n}&>(e^\gamma-\varepsilon)\log\log n\\
                   &=\frac{\log\log m}{\log\log m'}\left(1+\frac{\log m'/m}{\log m'\log\log m}\right)\left(e^\gamma+\frac{0.648214}{(\log\log m)^2}\right)\log\log n\\
                   &\geq\left(1+\frac{\log n/m}{\log n\log\log m}\right)\frac{\sigma(m)}{m},
\end{align*}
where the last inequality holds by (\ref{Rob 0.6482 m}) and (\ref{loglog m/loglog m'...=ey-eps}).
\end{proof}

Now we are going to state the main theorem of this paper which is the second step towards the delicacy of the RH, i.e.,
\begin{thm}\label{X'=infty}
The set $X'$ has infinite number of elements.
\end{thm}

\begin{proof}
If the RH is true, then the set $X'$ has infinite elements by (\ref{X<X'<S}). If RH is not true, then there exists $m_0\geq10080$ such that
$$
\frac{\sigma(m_0)/m_0}{\sigma(m)/m}>\frac{\log\log m_0}{\log\log m},\qquad\text{for all\ }\ m\geq10080.
$$
By Lemma \ref{s(n)/n>s(m)/m(1+log(n/m)/(log n loglog m))} there exists $m'>m_0$ such that $m'$ satisfies the inequality
$$
\frac{\sigma(m')/m'}{\sigma(m_0)/m_0}>1+\frac{\log m'/m_0}{\log m'\log\log m_0}.
$$
Let $n$ be the first number greater than $m_0$ which satisfies
$$
\frac{\sigma(n)/n}{\sigma(m_0)/m_0}>1+\frac{\log n/m_0}{\log n\log\log m_0}.
$$
Then $n\in X'$.
\end{proof}

\begin{defn}
Let $n_1=10080$. We find $n_2$ such that
\begin{equation*}\label{X' defn}
\frac{\sigma(n_2)/n_2}{\sigma(n_1)/n_1}>1+\frac{2\log n_2/n_1}{(\log n_2+\log n_1)\log\log n_1}.
\end{equation*}
Now we find $n_3$ such that
\begin{equation*}\label{X' defn}
\frac{\sigma(n_3)/n_3}{\sigma(n_2)/n_2}>1+\frac{2\log n_3/n_2}{(\log n_3+\log n_2)\log\log n_2},
\end{equation*}
and so on. We define $X''$ to be the set of all $n_1,\ n_2,\ n_3,\ldots$.
\end{defn}
Note that
$$
\#X=8150,\qquad\#X''=8187,\qquad\#X'=8378
$$
up to the $250\,000^{th}$ element of $S$ (we used the list of SA numbers tabulated in \cite{Noe}) and
$$
\#(X''\backslash X)=37,\qquad\#(X'\backslash X)=228.
$$
\section*{Acknowledgement}
 We would like to express our thanks to Professor Jean-Louis Nicolas  for his suggestions in the initial version. We would like to thank Muhammad Ali Khan and  Melinda Oroszl\'{a}nyov\'{a} for their comments and suggestions which rather improved the English presentation of the paper.

\end{document}